\documentclass{article}[12pt]
\usepackage{a4}
\usepackage[english]{babel}
\usepackage[T1]{fontenc}
\usepackage[all]{xy}
\usepackage{amssymb}
\usepackage{amsmath}
\usepackage{amsthm}
\usepackage{multirow}
\usepackage{graphicx}
\newtheorem{thm}{Theorem}[section]
\newtheorem{prop}[thm]{Proposition}

\newtheorem{defn}[thm]{Definition}
\newtheorem{expl}[thm]{Example}
\newtheorem{rmq}[thm]{Remark}

\numberwithin{equation}{section}

\newcommand{\lra}{\longrightarrow}

\newcommand{\as}{\mathcal{A}s}
\newcommand{\mcalo}{\mathcal{O}}
\newcommand{\opt}{\mathcal{O}^{\bullet}}

\newcommand{\dtwo}{\mathcal{D}_2}

\newcommand{\ims}{\mathcal{I}}
\newcommand{\crf}{\left\langle f \right\rangle}
\newcommand{\fitilde}{\widetilde{\varphi}}

\title{ \textbf{McClure-Smith cosimplicial machinery and  the cacti operad}}
\date{}

\author{ Paul Arnaud Songhafouo Tsopm\'en\'e}

\begin{document}
\maketitle

\begin{abstract}
McClure and Smith constructed a functor that sends a topological multiplicative operad $\mcalo$ to an $E_2$ algebra $\mbox{Tot} \opt$. They define in fact an  operad $\dtwo$ (acting on the totalization $\mbox{Tot} \opt$) weakly equivalent to the little $2$-disks operad. On the other hand, Salvatore showed that  $\dtwo$ is isomorphic to  the cacti operad $MS$, which has a nice geometric description. He also built a geometric action of  $MS$  on $\mbox{Tot} \opt$. In this paper we detail  the McClure-Smith action and the cacti action. Our main result says that they are  compatible in the sense that some squares must commute. 
\end{abstract}

\section{Introduction}

A \textit{multiplicative operad} is a topological nonsymmetric operad $\mcalo$ endowed with a morphism $\as \lra \mcalo$ from the associative operad $\as=\{*\}_{n \geq 0}$ to $\mcalo$. To any multipicative operad, McClure and Smith \cite[Section 10]{mcc_smith04} associate a cosimplicial space $\opt$. They also define an $E_2$ operad $\dtwo$ (see Definition~\ref{dtwo_definition}), which acts on the totalization $\mbox{Tot} \opt$. This gives a functor from the category of multiplicative operads to the category of $E_2$-algebras (for us, an \textit{$E_2$-algebra} is a topological space endowed with an action of an operad weakly equivalent to the little $2$-disks operad). 

On the other hand,  Salvatore \cite[Proposition 8.2]{sal10} shows that the cacti operad $MS$ (see Definition~\ref{cacti_operad_defn}), which has a nice geometric description, is isomorphic to $\dtwo$. He also proves \cite[Theorem 5.4]{sal10} that $MS$ acts on the totalization $\mbox{Tot} \opt$. The idea of his proof is as follows. To any cactus, elements $a_1^{\bullet}, \cdots, a_n^{\bullet} \in \mbox{Tot} \opt$, and $t \in \Delta^k$ , he associates a planar tree whose vertices (except the root and the leaves) are labelled by the entries of $\mcalo$. Using now the operad structure of $\mcalo$, he gets an operation $\theta \in \mcalo^k$. 

We will  explicitly construct $\theta$, without using trees, and our proof will be more combinatorial. In fact, from the same data as those of Salvatore, we  first associate a word (instead of a labelled tree) on the alphabet $\overline{n}=\{1, \cdots, n\}$. Next, by "suitable cutting" this word, we  obtain an explicit formula for $\theta$ (many illustrative examples are given). 

The natural question one can ask is whether the McClure-Smith action and the cacti action are equivalent. The following theorem, which is the main result of this paper, gives a positive answer.

\begin{thm} \label{main_thm}
Let $\opt$ be cosimplicial space associate to a multiplicative operad. Then the McClure-Smith  and the cacti actions on the totalization $\emph{Tot} \opt$ are  equivalent. More precisely, there is a square
$$
\xymatrix{MS(n) \times (\emph{Tot} \opt)^n \ar[d] \ar[r] & \emph{Tot} \opt \ar[d] \\
                    \dtwo(n) \times (\emph{Tot} \opt)^n \ar[r] & \emph{Tot} \opt }
$$
that commutes for each $n \geq 0$.
\end{thm}

\textbf{Outline of the paper.}

\begin{enumerate}
\item[-] In Section~\ref{basic_section} we recall the notions of cosimplicial spaces and totalizations. We also recall the  notion of multiplicative operads. 
 \item[-] In Section~\ref{cacti_section} we recall the cacti operad $MS$, and we construct an explicit action  on $\mbox{Tot} \opt$. 
 \item[-] In Section~\ref{dtwo_section} we first recall the McClure-Smith operad $\dtwo$. Then we explicit the details of its action on $\mbox{Tot} \opt$.
 \item[-] In Section~\ref{main_result_section} we prove Theorem~\ref{main_thm}.
\end{enumerate}

\section{Cosimplicial spaces, totalizations, and multiplicative operads} \label{basic_section}

This section recalls some basic notions. We also recall the fundamental construction $\mcalo \rightsquigarrow \opt$ due to McClure and Smith in \cite{mcc_smith04}.

Let us start with the notion of \textit{cosimplicial spaces} and \textit{totalizations}. Let $\Delta$ be the category whose objects are ordered sets on the form $[k]=\{0, \cdots, k\}$, $k \geq 0$, and morphisms are non decreasing maps. A \textit{cosimplicial space} is a covariant functor $\opt \colon \Delta \lra \mbox{Top}$ from $\Delta$ to topological spaces. One can define a cosimplicial space as a sequence $\opt=\{\mcalo^k\}_{k \geq 0}$ of topological spaces equipped with maps $d^i \colon \mcalo^k \lra \mcalo^{k+1}$ (called \textit{cofaces}) and $s^j \colon \mcalo^{k+1} \lra \mcalo^k$ (called \textit{codegeneracies}) satisfying some identities (known as \textit{cosimplicial relations}). One of the simplest examples of a cosimplicial space is the \textit{standard cosimplicial space} $\Delta^{\bullet}= \{\Delta^k\}_{k \geq 0}$. The space $\Delta^k$ is the standard geometric $k$-simplex. Throughout this paper it will be defined by
 $$
\Delta^k=\left\{
     \begin{array}{ll}
     * & \mbox{if}\;\; k=0 \\
     \{(t_1,\cdots,t_k) \in [-1,1]^k: -1 \leq t_1 \leq \cdots \leq t_k \leq 1\} & \mbox{if}\;\; k \geq 1. 
     \end{array}
\right.     
$$
The maps $d^i\colon \Delta^k \longrightarrow \Delta^{k+1}$ and $s^j\colon \Delta^{k+1} \longrightarrow \Delta^k$  are defined as follows. 
\begin{enumerate}
 \item[$\bullet$] For $0 \leq i \leq k+1$, for  $t=(t_1, \cdots, t_k) \in \Delta^k$, 
$$
  d^i(t)= \left\{
                                    \begin{array}{lll}
                                      (-1, t_1, \cdots, t_k) & \mbox{if  $i=0$}\\
                                   (t_1, \cdots,t_i, t_i, \cdots, t_k)& \mbox{if  $1 \leq i \leq k$ }\\
                                   (t_1, \cdots, t_k, 1) & \mbox{if  $i=k+1$}. \end{array}
                                     \right. $$
 \item[$\bullet$] For $1 \leq j \leq k+1, t=(t_1, \cdots, t_{k+1}) \in \Delta^{k+1}$,
$$
  s^j(t)=(t_1, \cdots, t_{j-1}, t_{j+1},  \cdots, t_{k+1}). 
$$
 \end{enumerate} 

\begin{defn}
The \emph{totalization} of a cosimplicial space $\opt$, denoted by $\emph{Tot} \opt$, is the space of natural transformations from the standard cosimplicial space $\Delta^{\bullet}$ to $\opt$. That is, 
$$\emph{Tot} \opt := \emph{Nat}(\Delta^{\bullet}, \opt).$$
\end{defn}

Let  us define now the important notion of \textit{multiplicative operads}.  Roughly speaking, an element of an operad is an operation with many inputs and one output. To be more precise, we have the following definition.

\begin{defn}  An \emph{operad} is a collection $\mcalo=\{\mcalo(k)\}_{k \geq 0}$ of topological spaces together with an unit $id \in \mcalo(1)$ and insertion maps
$$\circ_i \colon \mcalo(p) \times \mcalo(q) \lra \mcalo(p+q-1), 1 \leq i \leq p,$$
such that for $x \in \mcalo(p)$, $y \in \mcalo(q)$, $z \in \mcalo(k)$,
$$
\begin{array}{lll}
(x \circ_i y) \circ_{j+q-1} z & = & (x \circ_j z) \circ_i y \ \ \mbox{for $1 \leq i < j \leq p$} \\
x \circ_i (y \circ_j z) & = & (x \circ_i y) \circ_{i+j-1} z \ \ \mbox{for $1 \leq i \leq p$ and $1 \leq j \leq q$} \\
x \circ_i id & = & id \circ_1 x = x.
\end{array}
$$ 

\end{defn}

As said in the previous definition,  all our operads are topological and nonsymmetric. The simplest example of operad is the \textit{associative operad} $\as=\{\as(k)\}_{k \geq 0}$.  Recall that  each $\as(k)$ is the one point space $*$. In other words, $\as=\{*\}_{k \geq 0}$.

\begin{defn}
 A \emph{multiplicative operad} is a topological nonsymmetric  operad $\mcalo$ equipped with a map $\as \lra \mcalo$ of nonsymmetric operads from the associative operad $\as$ to $\mcalo$.
\end{defn}

The following remark gives an equivalent definition of a multiplicative operad. 

\begin{rmq} \label{multoperads-cosimplicialobjects2}
 A multiplicative structure on an operad $\mathcal{O}$ is equivalent to having special operations $e \in \mathcal{O}(0)$ and $\mu \in \mathcal{O}(2)$ satisfying 
$$\mu \circ_1 \mu=\mu \circ_2 \mu\ \ \ \mbox{and}\ \ \ \mu \circ_1 e=\mu \circ_2 e = id.$$ 
\end{rmq}

\begin{prop} \emph{\cite[Section 10]{mcc_smith04}}  \label{multoperads-cosimplicialobjects} 
To any multiplicative operad $\mcalo$, one can associate a cosimplicial space $\opt$. 
\end{prop} 

\begin{proof}   Let $\mcalo=\{\mcalo(k)\}_{k \geq 0}$ be a multiplicative operad, and let  $e \in \mcalo(0)$ and $\mu \in \mcalo(2)$ as in Remark~\ref{multoperads-cosimplicialobjects2}.  Define $\mathcal{O}^k=\mathcal{O}(k)$. Define also the cofaces morphisms $d^i\colon \mathcal{O}^k \longrightarrow \mathcal{O}^{k+1}$ and the codegeneracies morphims $s^j \colon \mathcal{O}^{k+1} \longrightarrow \mathcal{O}^k$ by the following formulas.
 \begin{enumerate}
 \item[$\bullet$] For $0 \leq i \leq k+1$, $x \in \mcalo^k$, define 
 $$ d^i(x)= \left\{
                                    \begin{array}{lll}
                                      \mu \circ_2 x & \mbox{if} & i=0 \\
                                  x \circ_i \mu & \mbox{if} &  1 \leq i \leq k \\
                                   \mu \circ_1 x & \mbox{if} &  i=k+1. \end{array}
                                     \right. $$
 \item[$\bullet$] For $1 \leq j \leq k+1$, $y \in \mcalo^{k+1}$, define  $s^j(y)= y \circ_j e$
\end{enumerate} 
It is straightforward to check cosimplicial relations  with $d^i$ and $s^j$ thus defined. 
\end{proof}

In the rest of this paper $\mcalo$ is a multiplicative operad, and $\opt$ is the associated cosimplicial space (as in Proposition~\ref{multoperads-cosimplicialobjects}).

\section{The cacti operad $MS$ and its action on $\mbox{Tot} \opt$} \label{cacti_section}

In this section we define the cacti operad $MS$ (see Definition~\ref{cacti_operad_defn}) and show that it explicitly acts  on $\mbox{Tot} \opt$ (Theorem~\ref{action_cacti_operad}).  In all this paper, for $n \geq 0$, the set $\overline{n}$ is defined by $\overline{n}=\{1, \cdots, n\}$.

We begin with the definition of the cacti operad $MS$. Let $S^1$ be the unit circle viewed as the quotient of the interval $I=[-1, 1]$ by the relation $-1 \sim 1$. Let $\pi \colon [-1, 1] \lra S^1$ denote  the canonical surjection, and let $*=\pi(1)$ denote the base point of $S^1$. To define $MS$ we need to define first a family of topological spaces 
$$\ims=\{\ims_k(n): n \geq 0 \ \mbox{and} \ 0 \leq k \leq \infty\}.$$ 
Let $n \geq 1$ be an integer. We start by defining the space $\ims_{\infty}(n)$. Next we will define $\ims_k(n)$ as a subspace of $\ims_{\infty}(n)$.\\
Let $K=\{K_i=[x_i^K,x_{i+1}^K]\}_{i=0}^{p_K-1}$ be a family of closed subintervals of $I$ satisfying the following two conditions:
\begin{itemize}
\item[$(P_1)  \colon $] $p_K \geq n$, $x_0^K=-1$, $x_p^K=1$ and the points $x_0^K, \cdots, x_p^K$ are pairwise distinct.
\item[$(P_2) \colon $] The intervals $K_i$ define $n$ $1$-manifolds $I_1^K, \cdots, I_n^K$ of disjoint interiors and with equal length. This means that each $I_j^K$ is an union of some $K_i$.
\end{itemize} 
We denote by $\mathcal{P}_n$ the collection of such a family $K$. That is,
\begin{eqnarray}
\mathcal{P}_n=\{K=\{K_i=[x_i^K,x_{i+1}^K]\}_{i=0}^{p_K-1} | \  K\ \mbox{satisfies} \ (P_1) \ \mbox{and} \ (P_2) \}.
\end{eqnarray}
The image of a $n$-tuple $(I_1^K, \cdots, I_n^K )$ (respectively the image of intervals $K_i$) under the canonical surjection $\pi$ will be denoted again by $(I_1^K, \cdots, I_n^K )$ (respectively by $\{K_i\}$). The set $\ims_{\infty}(n)$ is then defined by
$$\ims_{\infty}(n)=\{(I_1^K, \cdots, I_n^K)|\ K \in \mathcal{P}_n\}.$$
From now and in the rest of this paper, we will denote an element $x \in \ims_{\infty}(n)$  by $x=(I_1(x), \cdots, I_n(x))$ or just by $x=(I_1, \cdots, I_n)$. The family $K=\{K_i=[x_i^K,x_{i+1}^K]\}_{i=0}^{p_K-1}$ will be sometimes just denoted by $K=\{K_i=[x_i,x_{i+1}]\}_{i=0}^{p-1}$. \\
Let us equip  now  the set $\ims_{\infty}(n)$  with the following topology. Two elements $$x=(I_1(x), \cdots, I_n(x)) \qquad \mbox{and} \qquad y=(I_1(y), \cdots, I_n(y))$$ of $\ims_{\infty}(n)$ are said to be closed if the $1$-manifolds $I_i(x)$ and $I_i(y)$  are closed (in the sense that we have the inequality $\mbox{length}(I_i(x) \backslash \overset{\circ}{I}_i(y)) < \epsilon $ for some $\epsilon > 0$ too small) to each other for all $i$. Notice that $\ims_{\infty}(1)$ is the one point space.
In order to define the space $\ims_k(n)$ (for $k \geq 0$ be an integer), recall first the notion of the \textit{complexity} of a map. 

\begin{defn}
Let $T$ be a finite totally ordered set, $n$ be an integer, and  $f\colon T \lra \overline{n}=\{1, \cdots, n\}$ be a map. The \emph{complexity } of $f$, denoted by $\emph{cplx}(f)$, is defined as follows. 
\begin{enumerate}
\item[$\bullet$] If $n=0$ or $n=1$ then $\emph{cplx}(f)=0$. 
\item[$\bullet$] If $n=2$, let $\sim$ be the equivalence relation on  $T$ generated by
$$a \sim b\;\; \mbox{if $a$ is adjacent to $b$ and }\; f(a)=f(b).$$
The complexity of $f$ is equal to the number of equivalence classes minus $1$. 
\item[$\bullet$] If $n >2$, let $f_{ij} \colon f^{-1}(\{i,j\}) \lra \{i,j\}$ denote the restriction of $f$ on $f^{-1}(\{i, j\})$. The complexity of $f$ is equal to the maximum of complexities of the restrictions $f_{ij}$ as $\{i, j\}$ ranges over the two-element subsets of $\overline{n}$. That is,
\begin{eqnarray} \label{complexity_defn}
\emph{cplx}(f)=\underset{1 \leq i < j \leq n}{\emph{Max}}(\emph{cplx}(f_{ij})).
\end{eqnarray}
\end{enumerate}
\end{defn}

Note that a map $f \colon T \lra \{1, \cdots, n\}$ can be viewed as a word of length $\left|T\right|$ on the alphabet $\{1, \cdots, n\}$ (here $\left|T\right|$ denotes the cardinal of $T$). The \textit{length of the alphabet}  $\{1, \cdots n\}$ is defined to be its cardinal.   

\begin{expl}
\begin{enumerate}
\item[(a)] If $\left|T\right|=5$, $n=2$ and $f$ is defined by the word $f=12212$ then $\emph{cplx}(f)=3$.
\item[(b)] Assume that $\left|T\right|=8$,  $n=3$ and $f=31232113$, and consider the following tabular
\begin{center}
\begin{tabular}{| l | l | l |l |}
\hline
map & $f_{12}=12211$ & $f_{13}=313113$ & $f_{23}=32323$ \\ \hline
complexity & $\emph{cplx}(f_{12})=2$ & $\emph{cplx}(f_{13})=4$ & $\emph{cplx}(f_{23})=4$ \\
\hline
\end{tabular}
\end{center}
By this tabular and by (\ref{complexity_defn}), we deduce that $\emph{cplx}(f)=4$.
\end{enumerate}
\end{expl}

Let $x =(I_1, \cdots, I_n) \in \ims_{\infty}(n)$ defined by a partition $K=\{K_i\}_{i=0}^{p-1} \in \mathcal{P}_n$. Let $T_x$  be the set defined by  $T_x= \{0,1,\cdots, p-1\}$. For each $k \in T_x$, it is clear that there exixts a unique $i_k \in \{1, \cdots, n\}$ such that $K_k \subseteq I_{i_k}$ (this comes from the condition $(P_2)$ above). This defines a map
\begin{eqnarray} \label{mapf}
f_x \colon T_x \longrightarrow \{1, \cdots, n\}
\end{eqnarray}

\begin{defn}
The  \emph{complexity} of  an element $x \in \ims_{\infty}(n)$, denoted by $\emph{cplx}(x)$, is defined to be the complexity of the map  $f_x$. That is, 
$$\emph{cplx}(x)=\emph{cplx}(f_x).$$
\end{defn}

We are now ready to define $\ims_k(n)$.

\begin{defn} \label{cplxofx_defn}
For $k \geq 0$, the space $\ims_k(n)$  is the subspace of $\ims_{\infty}(n)$ defined by  
$$\ims_k(n)=\{x \in \ims_{\infty}(n)| \ \emph{cplx}(x) \leq k\}.$$
\end{defn}

The following figure is an element  of $\ims_2(4)$.

\begin{figure}[h]
\begin{center}
\includegraphics[scale=0.5]{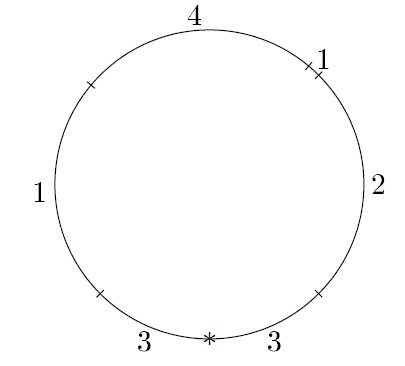}
\end{center}
\end{figure}

\begin{rmq}
The space $\ims_2(n)$ is a finite regular $CW$-complex with one cell for each $f_x(T_x)$ (see (\ref{mapf}) above for the definition of $f_x$). A cell labeled by some $f_x(T_x)$ is homeomorphic to $\prod_{i=1}^n \Delta^{\left|f_x^{-1}(i)\right|-1}$. For instance, the figure above is an element of a cell homeomorphic to  $\Delta^1 \times \Delta^1$, which is labeled by $314123$.
\end{rmq}




Now we want to define what so called \textit{cactus with $n$ lobes}. Let  $x=(I_1, \cdots, I_n) \in \ims_2(n)$ defined by a family of closed intervals $K=\{K_{i1}, \cdots, K_{il_i}\}_{i=1}^q \in \mathcal{P}_n$. Assume that  each $I_i$ is on the form  $I_i=\cup_{r=1}^{l_i}K_{ir}$, and define on $S^1$ the equivalence relation $\sim_i$ (for $1 \leq i \leq n$) generated by  
 $$(t_1 \sim_i t_2) \;\; \mbox{if and only if}\;\;  (t_1, t_2 \in K_{jr} \ \mbox{with} \   jr \notin \{i1, \cdots, il_i\}).$$
It is easy to see that the quotient of $S^1$ by this equivalence is homeomorphic to $S^1$. Let us denote by  $\pi_i$ the canonical surjection.
  $$\pi_i \colon S^1 \longrightarrow S^1/\sim_i \cong S^1.$$  
We thus construct a map $$c(x) \colon S^1 \longrightarrow (S^1)^n$$ 
defined by $c(x)=(\pi_1, \cdots, \pi_n)$, and called the \textit{cactus map}. Its image is called the \textit{cactus with $n$ lobes} associated to $x \in \ims_2(n)$. Let Coend$_{S^1}$ denote the coendomorphism operad  on $S^1$. Then there is an embedding
$$\overline{\tau}_n \colon \ims_2(n) \longrightarrow \mbox{Coend}_{S^1}(n)$$ defined by 
$$\overline{\tau}(x)=c(x).$$

\begin{rmq}
The collection of  spaces $\{\overline{\tau}_n(\ims_2(n))\}_{n \geq 0}$ is not a suboperad of $\emph{Coend}_{S^1}(\bullet)$. Indeed, let
 $x$ be an element of  $\ims_2(2)$ labaled by $212$.  Then $c(x)=(\pi_1, \pi_2) \in \emph{Coend}_{S^1}(2)$. Using now the operad structure of  $\emph{Coend}_{S^1}$, we get 
$$c(x) \circ_2 c(x)=(\pi_1, \pi_1 \circ \pi_2, \pi_2 \circ \pi_2)=(\pi_1, \mbox{constant map}, \pi_2 \circ \pi_2),$$ and it is impossible to find an element $z \in \mathcal{F}_2(3)$ such that $c(z)= (\pi_1, \mbox{constant map}, \pi_2 \circ \pi_2)$.
\end{rmq}

Since the collection $\{\overline{\tau}_n(\mathcal{F}_2(n))\}_{n \geq 0}$ is not far to be an operad, to get the right one, we introduce the space $\mbox{Mon}(I, \partial I)$ defined as follows. Let $\partial I=\{-1, 1\}$ denotes the boundary of $I$. Let $\varphi \colon I \lra I$ be a weakly monotone map such that its restriction on $\partial I$ coincides with the identity map $id_{\partial I}$. Then the map $\varphi$ passes to the quotient and gives a map $\fitilde \colon S^1 \lra S^1$, which is a typical element of $\mbox{Mon}(I, \partial I)$. 

\begin{rmq} \label{mon_is_contractible}
The space $\emph{Mon}(I, \partial I)$ is homeomorphic to the totalization $\emph{Tot} \Delta^{\bullet} \simeq *$. The homeomorphism $\emph{Mon}(I, \partial I) \stackrel{\cong}{\lra} \emph{Tot} \Delta^{\bullet}$ sends $\fitilde \in \emph{Mon}(I, \partial I)$ to $(t_1, \cdots, t_k) \longmapsto (\varphi(t_1), \cdots, \varphi(t_k))$.
\end{rmq}

 Considering the embedding  $$\tau_n \colon \ims_2(n) \times \mbox{Mon}(I, \partial I) \longrightarrow \mbox{Coend}_{S^1}(n)$$ defined by 
$$\tau_n(x, \fitilde)=c(x) \circ \fitilde \colon S^1 \lra (S^1)^n,$$
we have the following proposition.

\begin{prop} \label{suboperad_coend}
 The collection $\{\emph{im}(\tau_n)\}_{n \geq 0}=\{\tau_n(\ims_2(n) \times \mbox{Mon}(I, \partial I))\}_{n \geq 0}$ is a suboperad of $\mbox{Coend}_{S^1}$. 
\end{prop}

\begin{proof} The proof follows immediately from \cite[Proposition 4.5]{sal10}.
\end{proof}

Let $MS=\{MS(n)\}_{n \geq 0}$ be the collection of topological spaces defined by 
\begin{eqnarray} \label{ms(n)_definition}
 MS(n)=\left\{ \begin{array}{lll}
                  \ims_2(n) \times \mbox{Mon}(I, \partial I) & \mbox{if} & n \geq 1\\
									* & \mbox{if} & n=0.
                 \end{array} \right.
\end{eqnarray}
By transferring the operad structure of $\{\mbox{im}(\tau_n)\}_{n \geq 0}$ (given by Proposition~\ref{suboperad_coend}) on $MS$ via embeddings $\tau_n$, we endow $MS$ with an operad structure. 

\begin{defn} \label{cacti_operad_defn}
The operad $MS$ is called the \textit{cacti operad}.
\end{defn}

\begin{prop} \label{ms2_s1}
There exists a weak equivalence $(\phi, id) \colon S^1 \stackrel{\sim}{\lra} MS(2)$.
\end{prop}

\begin{proof}
Define  $\phi \colon S^1 \lra \ims_2(2)$ by 
\begin{eqnarray} \label{s1_i22}
\phi(\tau)=\left\{\begin{array}{lll}
                    (I_1=K_0 \cup K_2, I_2=K_1), K_0=[-1, \tau], K_1=[\tau, 1+\tau], K_2=[1+\tau, 1] & \mbox{if} & -1 < \tau < 0 \\ 
										(I_1=K_1, I_2=K_0 \cup K_2), K_0=[-1, \tau-1], K_1=[\tau-1, \tau], K_2=[\tau, 1] & \mbox{if} & 0 < \tau < 1 \\
										(I_1=K_0, I_2=K_1), K_0=[-1, 0], K_1=[0, 1] & \mbox{if} & \tau=0 \\
										(I_1=K_1, I_2=K_0), K_0=[-1, 0], K_1=[0, 1] & \mbox{if} & \tau=\pm 1 
                   \end{array} \right.
\end{eqnarray}
It is not difficult to see that $\phi$ is a homeomorphism. Therefore the map 
\begin{eqnarray} \label{mapphi-id}
(\phi, id) \colon S^1 \lra MS(2)=\ims_2(2) \times \mbox{Mon}(I, \partial I),
\end{eqnarray}
 where $id$ is the map that sends each point of $S^1$ to the identity map $id_{S^1} \colon S^1 \lra S^1$,  is a weak equivalence since the space $\mbox{Mon}(I, \partial I)$ is contractible by Remark~\ref{mon_is_contractible}.
\end{proof}
 
The following theorem is originally due to P. Salvatore in \cite{sal10}. He gives a nice geometric proof to it.  Here we furnish another proof, which is more combinatorial. We provide in fact explicit formulas of the action of $MS$ on $\mbox{Tot} \mcalo^{\bullet}$.

\begin{thm} \emph{\cite[Theorem 5.4]{sal10}} \label{action_cacti_operad}
Let $\opt$ be  a cosimplicial space defined by a multiplicative operad $\mcalo$. Then the cacti operad $MS$ acts on the totalization $\emph{Tot} \mcalo^{\bullet}$.
\end{thm}

\begin{proof} 
Let $(x, \fitilde) \in MS(n)=\ims_2(n) \times \mbox{Mon}(I, \partial I)$, let $a_i^{\bullet} \in \mbox{Tot} \mathcal{O}^{\bullet}$, $1 \leq i \leq n$. Our aim is to construct from these data a family
\begin{eqnarray}
\theta_n((x, \fitilde), (a_1^{\bullet}, \cdots, a_n^{\bullet}))_k \colon \Delta^k\longrightarrow \mathcal{O}^k, k \geq 0,
\end{eqnarray}
of maps that commute with cofaces and codegeneracies. 

Let $t= (-1=t_0 \leq t_1 \leq \cdots \leq t_k \leq t_{k+1}=1)$ be an element of $\Delta^k$. Define a family of closed intervals $\{J_j\}_{j=0}^k$ by setting $J_j=[\varphi(t_j), \varphi(t_{j+1})]$.  Assume that $x \in \ims_2(n)$ is defined by a family of $p$ closed intervals $K=\{[x_i, x_{i+1}]\}_{i=0}^{p-1}$, and consider the set 
$$E=\{\varphi(t_j)| \ 0 \leq j \leq k+1 \} \cup \{x_i| \ 0 \leq i \leq p\}.$$
 Define now a family $\{L_l=[a_l, b_l]\}_{l=0}^m$ of closed subintervals of $I$ as follows. 
\begin{enumerate}
\item[-] each $a_l$ or $b_l$  belongs to $E$;
\item[-] the interior of each $L_l$ contains no element of $E$; 
\item[-] $\cup_{l=0}^m L_l=I$ and 
\begin{eqnarray} \label{m_equation}
m=k+p-1.
\end{eqnarray}
\end{enumerate}
  
\begin{center} \textbf{(a) Assume that for all $i$ and for all $j$, $x_i \neq \varphi(t_j)$} \end{center}

In this case  there exists, for each $l \in \{0, \cdots, m\}$, an unique element $i_l \in \{1,\cdots, n\}$ and an unique element $j_l \in \{0, \cdots, k\}$ such that $L_l \subseteq I_{i_l}$ and $L_l \subseteq J_{j_l}$. This gives two maps
\begin{eqnarray} \label{htf}
\xymatrix{[k] & [m] \ar[l]_-{h} \ar[r]^-{f} & \overline{n}}
\end{eqnarray}
defined by 
$$h(l)=j_l \qquad \mbox{and} \qquad f(l)=i_l.$$ 
It is easy to see that the map $h$ is a morphism in the simplicial category $\Delta$. It is also easy to see that $f$ is a surjective map, and its complexity is less than or equal to $2$ (this is because we have taken $x$ in $\ims_2(n)$, and by Definition~\ref{cplxofx_defn} we have $\mbox{cplx(x)} \leq 2$). 

Let $i \in \overline{n}$. We want to explicitly construct an element $y _i \in \Delta^{f^{-1}(i)}=  \Delta^{\left|f^{-1}(i)\right|-1}$. Let us set 
$$I_i= \cup_{j=0}^{l_i} K_{i_j}, K_{i_j}=[x_{i_j}, x_{i_j+1}] \subseteq [-1, 1], \mbox{and}\ x_{i_{l_i+1}}=1.$$
Define by induction a family $\{g_{i_q} \colon [-1,1] \lra [-1,1]\}_{q=0}^{l_i+1}$ of maps as follows. 

\begin{eqnarray} \label{gio}
g_{i_0}(z)=\left\{ \begin{array}{lll}
                    -1 & \mbox{if} & z \in [-1, x_{i_0}] \\
										z-(x_{i_0}+1) & \mbox{if} & z > x_{i_0},
                 \end{array} \right.  
\end{eqnarray}
and for $0 \leq q \leq l_i$,  
\begin{eqnarray} \label{giq}
g_{i_{q+1}}(z)=\left\{ \begin{array}{lll}
                    z & \mbox{if} & z \in [-1, g_{i_q} \circ \cdots \circ g_{i_0} (x_{i_q+1})] \\
										g_{i_q} \circ \cdots \circ g_{i_0} (x_{i_q+1}) & \mbox{if} & z \in [g_{i_q} \circ \cdots \circ g_{i_0} (x_{i_q+1}), g_{i_q} \circ \cdots \circ g_{i_0} (x_{i_{q+1}})] \\
										z-(x_{i_{q+1}}-x_{i_q+1}) & \mbox{if} & z > g_{i_q} \circ \cdots \circ g_{i_0} (x_{i_{q+1}})
                 \end{array} \right.  
\end{eqnarray}
Intuitively, the map  $g_{i_0}$ contracts the interval $[-1, x_{i_0}]$ to $-1$, and moves other points by the translation of vector $-(x_{i_0}+1)$, the map $g_{i_1}$ contracts the interval $[g_{i_0}(x_{i_0+1}), g_{i_0}(x_{i_1})]$ to $g_{i_0}(x_{i_0+1})$, and moves other points by the translation of vector $-(x_{i_1}-x_{i_0+1})$, and so on. At the end of this process, we obtain an interval of length $\frac{2}{n}$. More precisely, if we define $g$ to be the composite 
\begin{eqnarray} \label{mapg}
g=g_{i_{l_i+1}} \circ g_{i_{l_i}} \circ \cdots \circ g_{i_1} \circ g_{i_0},
\end{eqnarray}
 then
$$g([-1,1])=[-1, \frac{2-n}{n}].$$
Define also a map $\alpha \colon [-1, \frac{2-n}{n}] \lra [-1, 1]$ by 
\begin{eqnarray} \label{mapalpha}
\alpha(z)=nz+n-1.
\end{eqnarray}
Notice the map $\alpha$ fixes $-1$, and sends $\frac{2-n}{n}$ to $1$. In fact $\alpha$ allows to rescale the interval $[-1, \frac{2-n}{n}]$. Consider now the map $\widetilde{g} \colon [-1, 1] \lra [-1,1]$ defined by 
$$
\widetilde{g}= \alpha \circ g.
$$
For $j \in \{0, \cdots, l_i \}$, if $t_{i_j}^1, \cdots, t_{i_j}^{v_j}$ denote elements of the set $\{\varphi(t_1), \cdots, \varphi(t_k)\}$ that belong to $K_{i_j}$, then $y _i \in  \Delta^{\left|f^{-1}(i)\right|-1}$ is defined by
\begin{eqnarray} \label{yi_definition}
y_i=(\widetilde{g}(t_{i_0}^1), \cdots, \widetilde{g}(t_{i_0}^{v_0}), \widetilde{g}(x_{i_0+1})),  \widetilde{g}(t_{i_1}^1), \cdots, \widetilde{g}(t_{i_1}^{v_1}), \cdots, \widetilde{g}(x_{i_{l_i-1}+1}), \widetilde{g}(t_{i_{l_i}}^1), \cdots, \widetilde{g}(t_{i_{l_i}}^{v_{l_i}})).
\end{eqnarray}
An illustration for $y_i$ is given in the first part of Example~\ref{cacti_action_expl}. 


Let us construct now by induction on $n$ the operation $\theta_n(((x, \fitilde), (a_1^{\bullet}, \cdots, a_n^{\bullet}))_k(t) \in \mathcal{O}(k)$. The map $f$ will be thought as a word of length $m+1$ on the alphabet $\{1, \cdots, n\}$. If $W$ is a word on an alphabet of length $*$, we will write $\theta_*(W)$ for the associated operation. For instance, the operation $\theta_n(((x, \fitilde), (a_1^{\bullet}, \cdots, a_n^{\bullet}))_k(t)$ will be sometimes denoted by $\theta_n(f)$.  

If $n=1$ then $c(x)=id_{S^1}$. Define $\theta_1((x, \fitilde), a_1^{\bullet})(t) \in \mcalo(k)$ by
$$\theta_1((x, \fitilde), a_1^{\bullet})(t)=a_1^{k}(t)$$
If $n=2$, let $i, j$ be two distinct elements inside  $\{1,2\}$. Since the complexity of $f$ is less than or equal to $2$, there are two possibilities for writing the word $f$.
\begin{enumerate}
 \item[$\bullet$] Assume that the map $f \colon [m] \longrightarrow \{1,2\}$ is on the form $f=\underbrace{ i \cdots i}_{r+1} \underbrace{j \cdots j}_{s+1}$. \\
This implies that we have exactly two closed intervals $K_0, K_1$ defining $x=(I_1, I_2) \in \ims_2(2)$, and therefore, $p=2$ (recall that $p$ is the number of intervals $K_i$ defining $x \in \ims_2(2)$). We claim that $r+s=k$. Indeed, since the length of the word $f$ is equal to $m+1$, it follows that
$$
\begin{array}{lll}
(r+1)+(s+1) &= & m+1 \\
            & = & k+p \ \  \mbox{since $m=k+p-1$ by (\ref{m_equation}) above} \\
						& = & k+2 \ \ \mbox{since $p=2$.}
\end{array}
$$							
Let  $\mu \in \mcalo(2)$ denote the multiplication. Define $\theta_2((x, \fitilde), (a_1^{\bullet}, a_2^{\bullet}))_k(t) \in \mcalo(r+s)=\mcalo(k)$ by the formula 
\begin{eqnarray} \label{theta2_defn1}
\theta_2((x, \fitilde), (a_1^{\bullet}, a_2^{\bullet}))_k(t) = \theta_2(i\cdots i j \cdots j)= \mu(a_i^r(y_i), a_j^s(y_j)). 
\end{eqnarray}

\item[$\bullet$] Now we assume that the word $f$ is on the form $f=\underbrace{i \cdots i}_{r_1} \underbrace{j \cdots j}_{s+1} \underbrace{i \cdots i}_{r_2}$.\\
This implies that $p=3$. Like before, we can check that $r_1+r_2-1 +s-1=k$. Define the operation $\theta_2((x, \fitilde), (a_1^{\bullet}, a_2^{\bullet}))_k(t) \in \mathcal{O}(r_1+r_2-1+s-1)=\mcalo(k)$ by
\begin{eqnarray} \label{theta2_defn2}
\theta_2((x, \fitilde), (a_1^{\bullet}, a_2^{\bullet}))_k(t)=  \theta_2(i \cdots i j \cdots j i \cdots i)= a_i^{r_1+r_2-1}(y_i) \circ _{r_1} a_j^s(y_j).
\end{eqnarray}
\end{enumerate}
Let $n \geq 3$. For $i \leq n-1$, the operation $\theta_{i}((x, \fitilde), (a_1^{\bullet}, \cdots, a_i^{\bullet}))_k$ will be denoted just by $\theta_*(-)$. Assume  that  $\theta_*(W)$ is constructed for each word $W$ on an alphabet of length $* \leq n-1$. We want to construct $\theta_{n}(f)$.  Set $f(0)=i_0$ and define the integer 
$$m_0=\mbox{Max} \{j \in [m]| \ f(j)=i_0\}.$$ 
There are two possibilities depending of the fact that the word $f$ ends by the letter $i_0$ or not.
\begin{enumerate}
\item[$\bullet$] If $m_0=m$ then the map $f$ is on the form 
$$f=\underbrace{i_0 \cdots i_0}_{r_1} b_{11} \cdots b_{1s_1} \underbrace{i_0 \cdots i_0}_{r_2} b_{21} \cdots b_{2s_2} \cdots \underbrace{i_0 \cdots i_0}_{r_q} b_{q1} \cdots b_{qs_q}\underbrace{i_0 \cdots i_0}_{r_{q+1}},$$
with $b_{js} \neq i_0$ for all $j$ and for all $s$, and with  $r_1 + \cdots +r_{q+1}$ copies of $i_0$. Let  us set  
$$u = \left(\sum_{i=1}^{q+1} r_i \right)-1 \qquad \mbox{and} \qquad v=\sum_{i=1}^q r_i.$$
Define $\theta_{n}(f)$ by the formula
\begin{eqnarray} \label{thetan_formula}
\theta_{n}(f) = (((a_{i_0}^{u}(y_{i_0}) \circ_{v} \theta_*(b_{q1} \cdots b_{qs_q})) \circ_{v-r_q}  \cdots ) \circ_{r_1} \theta_*(b_{11} \cdots b_{1s_1}).
\end{eqnarray}
A perfect illustration for this formula is given by Example~\ref{cacti_action_expl2} below. 
\item[$\bullet$] If $m_0 < m$ then the map f is on the form 
$$f=i_0 \cdots i_0 b_1 \cdots b_w i_0 \cdots i_0 c_1 \cdots c_s$$
with $c_j \neq i_0$ for all $j$. Let $r$ be the number of letters (in the alphabet $\{1, \cdots, n\}$) appearing in the word $f(\{0, \cdots, m_0\})=i_0 \cdots i_0 b_1 \cdots b_w i_0 \cdots i_0$. Then, since each letter of the word $b_1 \cdots b_w$ does not appear in the word $c_1 \cdots c_s$  (because the complexity of $f$ is less than or equal to $2$), there is exactly $n-r$ letters appearing in the word $c_1 \cdots c_s$. Define
\begin{eqnarray} \label{thetan_formula2}
\theta_{n}(f)= \mu(\theta_r(i_0 \cdots i_0 b_1 \cdots b_w i_0 \cdots i_0), \theta_{n-r}(c_1 \cdots c_s)),
\end{eqnarray}
\end{enumerate}

\begin{center} \textbf{(b) Now we assume that there exists some integers $i$ and $j$ such that $x_i=\varphi(t_j)$} \end{center}

In this case, there is a finite number of  possibilities (depending of the fact that we consider the interval $[\varphi(t_j), x_i]$ or $[x_i, \varphi(t_j)]$ in the family $\{L_l\}_{l=0}^m$) to define the word $f$ . It is not difficult to show, by using the naturality of the maps $a_r^{\bullet}\colon \Delta^{\bullet} \longrightarrow \mathcal{O}^{\bullet}$ and the fact that $\mu(\mu,id)=\mu(id, \mu)$ (we have this equality because $\mcalo$ is a multiplicative operad),  that all these possibilities lead to the same element $\theta_n((x, \fitilde), (a_1^{\bullet}, \cdots, a_n^{\bullet}))_k(t) \in  \mathcal{O}(k)$. A good illustration of that is given by the second part of Example~\ref{cacti_action_expl} below.

We can check that the maps $\theta_n((x, \fitilde), (a_1^{\bullet}, \cdots, a_n^{\bullet}))_k \colon \Delta^k \longrightarrow \mathcal{O}^k$  thus defined are continuous and commute with  cofaces and codegeneracies. The continuity  comes essentially from the fact that $\mu(\mu,id)=\mu(id, \mu)$. The continuity of the map $\theta_n \colon MS(n) \times (\mbox{Tot} \mathcal{O}^{\bullet})^n \longrightarrow \mbox{Tot} \mathcal{O}^{\bullet}$ also comes from the same fact.  
\end{proof}

\begin{center}
\begin{tabular}{cc}
\includegraphics[scale=0.5]{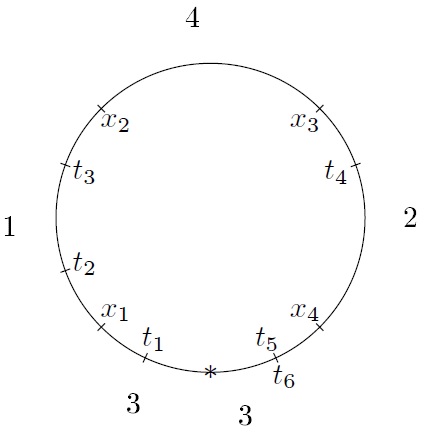}

&

\includegraphics[scale=0.5]{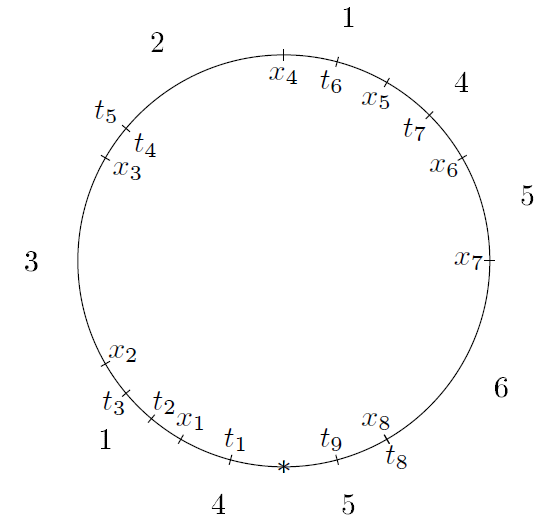}

\\

Figure a

&

Figure b
\end{tabular}

\end{center}

\begin{expl} \label{cacti_action_expl}
\begin{itemize}
\item[(a)] Let $n=4$, $x \in \ims_2(4)$ and $t=(t_1, \cdots, t_6) \in \Delta^6$ (see Figure a). Assume that $\fitilde=id_{S^1}$. Then $k=6$, $p=5$, $m=k+p-1=10$ and the map $f \colon [10] \longrightarrow \{1,2,3,4\}$ is defined by the word $f=33111422333$. Now let us explicitly define $y_1 \in \Delta^{\left|f^{-1}(1)\right|}= \Delta^2$. First of all, we have $I_1=[x_1, x_2]$. So the map $g$ (see (\ref{mapg})) is just equal to $g_{i_0}$ (notice that here $i_0=1$), and by the definition of this latter map (see (\ref{gio})) we have $g_{i_0}(z)=z-(x_1+1)$ for each $z \in [x_1, x_2]$. On the other hand, the map $\alpha$ here is defined by $\alpha(z)=4z+3$ (see  (\ref{mapalpha})). Therefore the image of each $z \in [x_1, x_2]$ under the composite $\widetilde{g} = \alpha \circ g_{i_0}$ gives $4z-4x_1-1$. Hence, 
$$y_1=(4t_2-4x_1-1, 4t_3-4x_1-1) \in \Delta^2$$ 
A similar computation gives $y_2=4t_4-4x_3-1 \in \Delta^1$. For $y_3$ we use the formula (\ref{yi_definition}), and we obtain 
$$y_3=(4t_1+3, 4x_1+3, 4t_5-4x_4+4x_1+3, 4t_6-4x_4+4x_1+3) \in \Delta^4.$$ 
Now we can define $\theta_4(f)$.
$$
\begin{array}{lll}
\theta_4(f) &= & a_3^4(y_3) \circ_2 \theta_3(111244) \; \; \mbox{by (\ref{thetan_formula})} \\
            & = & a_3^4(y_3) \circ_2 \mu(a_1^2(y_1), \theta_2(244)) \; \; \mbox{by (\ref{thetan_formula2})} \\
						& = & a_3^4(y_3) \circ_2 \mu(a_1^2(y_1), \mu(a_4^0(*), a_2^1(y_2))) \in \mcalo(6) \; \; \mbox{by (\ref{theta2_defn1})}.
\end{array}
$$
\item[(b)] Let $n=6$, $x \in \ims_2(6)$ and $t=(t_1, \cdots, t_9) \in \Delta^9$ (see Figure b above). Assume that  $\fitilde=id_{S^1}$. Then  $k=9$, $p=9$ and $m=k+p-1=17$. Since $t_8=x_8$, it follows that there are two possibilities to define the map $f \colon [17] \longrightarrow \{1,2,3,4,5,6\}$. \\
If  $f=441113222114456655$ then we have (by applying formulas (\ref{thetan_formula2}), (\ref{thetan_formula}), (\ref{theta2_defn1}) and (\ref{theta2_defn2}))
$$\theta_6(f)= \mu(a_4^3(y_4) \circ_2 (a_1^4(y_1) \circ_3 \mu(a_3^0(*), a_2^2(y_2))), a_5^2(y_5) \circ_1 a_6^1(y_6)) \in \mathcal{O}(9).$$
Here $y_5=(-1 \leq y_{51} \leq y_{52} \leq 1)$ is an element of $\Delta^2$. Let us denote it by $y_5^2$.

If $f=441113222114456555$ then we have (again by formulas (\ref{thetan_formula2}), (\ref{thetan_formula}), (\ref{theta2_defn1}) and (\ref{theta2_defn2})) 
$$\theta_6(f)= \mu(a_4^3(y_4) \circ_2 (a_1^4(y_1) \circ_3 \mu(a_3^0(*), a_2^2(y_2))), a_5^3(y_5) \circ_1 a_6^0(*)) \in \mathcal{O}(9).$$
Here $y_5=(-1 \leq y_{50} \leq y_{51} \leq y_{52} \leq 1)$ is an element of $\Delta^3$ with $y_{50}=y_{51}$. Let us denote it by $y_5^3$.

To check that these two possibilities lead to the same operation in $\mathcal{O}(9)$, it suffices to check that
$$a_5^3(y_5^3) \circ_1 a_6^0(*)= a_5^2(y_5^2) \circ_1 a_6^1(y_6).$$
Here we go
$$\begin{array}{lll}
a_5^3(y_5^3) \circ_1 a_6^0(*) & = & a_5^3(y_{50}, y_{51}, y_{52}) \circ_1 a_6^0(*)\;\;\mbox{because}\;\;y_5^3=(y_{50}, y_{51}, y_{52})\;\; \mbox{with}\;\;y_{50}=y_{51}  \\
  & = &  a_5^3(d^1(y_{51}, y_{52})) \circ_1 a_6^0(*)\;\;\mbox{because}\;\;(y_{51}, y_{51}, y_{52})= d^1(y_{51}, y_{52}), d^1 \colon \Delta^2 \lra \Delta^3\\
	& = & (d^1(a_5^2(y_{51}, y_{52}))  \circ_1 a_6^0(*) \;\;\mbox{by the naturality of}\;\; a_5^{\bullet} \\
& = & (a_5^2(y_{51}, y_{52}) \circ_1 \mu) \circ_1 a_6^0(*)\;\; \mbox{by the definition of the coface map $d^1 \colon \mcalo^2 \lra \mcalo^3$} \\
  & = &  a_5^2(y_{51}, y_{52}) \circ_1 (\mu \circ_1 a_6^0(*))\;\;\mbox{by the associativity in $\mcalo$} \\
	& = &  a_5^2(y_{51}, y_{52}) \circ_1 (d^1(a_6^0(*)))\;\; \mbox{by the definition of the coface map $d^1 \colon \mcalo^0 \lra \mcalo^1$} \\
  & = & a_5^2(y_{51}, y_{52}) \circ_1 a_6^1(d^1(*)) \;\;\mbox{by the naturality of}\;\; a_6^{\bullet}\\
  & = & a_5^2(y_5^2) \circ_1 a_6^1(y_{6}) .

\end{array}
$$

\end{itemize}
\end{expl}

\begin{expl} \label{cacti_action_expl2} 
In this example, we are in the case (a) of the proof of Theorem~\ref{action_cacti_operad}. Take $t=(t_1, \cdots, t_{14}) \in \Delta^{14}$ such that $t_i \neq t_j$ whenever $i \neq j$, and take $f=1112244221111333555511$. The operation $\theta_5(f) \in \mcalo(14)$ is then defined by
$$\begin{array}{lll}
\theta_5(f) & = & (a_1^8(y_1) \circ_7 \theta_2(3335555)) \circ_3 \theta_2(224422) \; \; \mbox{by (\ref{thetan_formula})} \\
            & = & (a_1^8(y_1) \circ_7 \mu(a_3^2(y_3), a_5^3(y_5))) \circ_3 (a_2^3(y_2) \circ_2 a_4^1(y_4)) \; \; \mbox{by (\ref{theta2_defn1}) and (\ref{theta2_defn2})}.
	\end{array}		
$$				
\end{expl}

\section{McClure-Smith operad $\dtwo$ and its action on $\mbox{Tot} \opt$} \label{dtwo_section}

Here we recall the construction of the operad $\dtwo$. We also give the details of its action on $\mbox{Tot} \opt$ (notice that this action was built by McClure and Smith in \cite{mcc_smith04}). We will write $\Delta_+$ for the category $\Delta \cup \{\emptyset\}$, and a covariant functor $X^{\bullet} \colon \Delta_+ \longrightarrow \mbox{Top}$ will be called an \textit{augmented cosimplicial space}. In all this section, the standard cosimplicial space $\Delta^{\bullet}$ will be viewed as an augmented cosimplicial space with $\Delta^{\emptyset}=\emptyset$. 

Let us begin with the definition of the augmented cosimplicial space 
$$\Xi_n^2(X^{\bullet}, \cdots, X^{\bullet}) \colon \Delta_+ \longrightarrow \mbox{Top}$$
in which we have $n$ copies of $X^{\bullet}$.

Let $\overline{n}$ as in the previous section. Define $Q_n$ to be the category whose  objects are pairs $(T,f)$, where $T$ is an object of the category $\Delta_+$ and $f \colon T \longrightarrow \overline{n}$ is a morphism in sets. A morphism from $(T,f)$ to $(T',f')$ consists of a morphism $g \colon T \longrightarrow T'$ in $\Delta_+$  such that $f=f'g$. Define also  $Q_n^2$ to be the full subcategory of $Q_n$ consisting of  pairs $(T,f)$ such that $\mbox{cplx}(f) \leq 2$. Consider now the diagram
 $$ \xymatrix{
      Q_n^2 \ar@{^{(}->}[r]^{\rho_2} \ar[d]_{\phi_n} & Q_n  \ar[r]^-{\psi_n} & \mbox{Top}^n \ar[r]^-{\prod_n} & \mbox{Top}\\
      \Delta_+ 
    } $$
in which 
\begin{enumerate}
\item[-] $\rho_2$ is the inclusion functor,
\item[-] $\psi_n$ is defined to be $\psi_n(T,f)=(X^{f^{-1}(1)},\cdots , X^{f^{-1}(n)})$,
\item[-] $\prod_n$ is the product functor and 
\item[-] $\phi_n$ is the projection on the first component.
\end{enumerate}
The covariant functor $\Xi_n^2(X^{\bullet}, \cdots, X^{\bullet}) \colon \Delta_+ \longrightarrow \mbox{Top}$ is defined 
to be the left Kan extension of the composite  $\prod_n \circ \psi_n \circ \rho_2$ along  $\phi_n$. By the definition of a left Kan extension, the functor $\Xi_n^2(X^{\bullet}, \cdots, X^{\bullet})$ is explicitly defined as follows. 

Let  $S$ be an object of the category $\Delta_+$. We want to define the space $\Xi_n^2(X^{\bullet}, \cdots, X^{\bullet})(S)$ associated to $S$.  First we define  the category $A_{nS}$ whose objects  are triples $(h, T, f)$ where $T$ is an object in $\Delta_+$, $h\colon T \longrightarrow S$ is a morphism in $\Delta_+$, and $(T,f)$ is an object in $Q_n^2$. A morphism from $(h, T, f)$ to $(h', T', f')$ consists of a morphism $g\colon T \longrightarrow T'$ such that $h=h'g$ and $f=f'g$. We will sometimes write $S \stackrel{h}{\longleftarrow} T \stackrel{f}{\longrightarrow} \overline{n}$ for an object $(h, T, f)$ of the category $A_{nS}$. Next we define the functor $p_n\colon A_{nS} \longrightarrow Q_n^2$ by $p_n(h,T,f)= (T,f)$, and we consider the $A_{nS}$-diagram 
$$F_{nS}= \prod_n \circ \psi_n \circ \rho_2 \circ p_n \colon A_{nS} \longrightarrow \mbox{Top}.$$ 
The space $\Xi_n^2(X^{\bullet}, \cdots, X^{\bullet})(S)$ is then explicitly defined to be the colimit of $F_{nS}$. That is,  
\begin{eqnarray} \label{xi2_definition}
 \Xi_n^2(X^{\bullet}, \cdots, X^{\bullet})(S)=\underset{A_{nS}}{\mbox{colim}} F_{nS}.
\end{eqnarray}  
Notice that $\Xi^2_0(X^{\bullet}, \cdots, X^{\bullet})(S)$ is the one point space since each Cartesian product $\underset{i \in \overline{0}}{\prod} f^{-1}(i)$ is a point (because $\overline{0}=\emptyset$).\\  
On morphisms, the functor  $\Xi_n^2(X^{\bullet}, \cdots, X^{\bullet}) \colon \Delta_+ \lra \mbox{Top}$ is defined in the obvious way.  We are now ready to define the operad $\dtwo$.

\begin{defn} \label{dtwo_definition} 
For $n \geq 0$ the space $\dtwo(n)$ is defined to be 
$$\mathcal{D}_2(n)=\emph{Tot} (\Xi_n^2(\Delta^{\bullet},\cdots, \Delta^{\bullet})) = \emph{Nat}(\Delta^{\bullet}, \Xi_n^2(\Delta^{\bullet},\cdots, \Delta^{\bullet})).$$
\end{defn}

McClure and Smith show \cite[Section 9]{mcc_smith04} that the collection $\{\dtwo(n)\}_{n \geq 0}$ is a topological operad. They also show \cite[Theorem 9.1 (a)]{mcc_smith04} that this operad is actually weakly equivalent in the category of operads to the  operad $B_2$ of little $2$-disks. 

\begin{defn} \label{xi2_algebra_defn}  An augmented cosimplicial space $X^{\bullet}$ is a \emph{$\Xi^2$-algebra}  is endowed with a \emph{$\Xi^2$-structure} if there is a family
$$\{\Theta_n \colon \Xi_n^2(X^{\bullet},\cdots, X^{\bullet}) \lra X^{\bullet} \}_{n \geq 0}$$
of natural transformations satisfying axioms (a), (b) and (c) of \cite[Definition 4.5]{mcc_smith04} with $\mathcal{F}$ replaced by $\Xi^2$.   
\end{defn}

\begin{defn} An augmented cosimplicial space $X^{\bullet}$ is said to be \emph{reduced} if $X^{\emptyset}$ is the one point space. 
\end{defn}



\begin{prop} \emph{\cite[Proposition 10.3]{mcc_smith04}} \label{mcc_smith_proposition}
A sequence $\mathcal{O}=\{\mcalo(n)\}_{n \geq 0}$ of topological spaces is endowed with a structure of multiplicative operad if and only if the associated reduced augmented cosimplicial space $\opt$ is a $\Xi^2$-algebra. 
\end{prop}

In \cite[Theorem 9.1 (b)]{mcc_smith04}$\dtwo$, McClure and Smith prove that the operad $\dtwo$ acts on $\mbox{Tot} \opt$, when the reduced augmented cosimplicial space $\opt$ is built from a multiplicative operad $\mcalo$. We now recall this action. To do that,  we will define (for each $n \geq 0$) a map 
$$\beta_n \colon \dtwo(n) \times (\mbox{Tot} \opt)^n \lra \mbox{Tot} \opt.$$
Let $\alpha=\{\alpha_k\}_{k \geq 0} \in \dtwo(n)$, and let $(a_1^{\bullet},\cdots, a_n^{\bullet}) \in (\mbox{Tot} \opt)^n$. We want to define $\beta_n(\alpha, (a_1^{\bullet}, \cdots, a_n^{\bullet})) \in \mbox{Tot} \opt$.  Form the diagram  
$$\xymatrix{\Delta^{\bullet} \ar[r]^-{\alpha} & \Xi^2_n(\Delta^{\bullet}, \cdots, \Delta^{\bullet}) \ar[rr]^-{\prod_{i=1}^n a_i^{\bullet}} & & \Xi^2_n(\opt, \cdots, \opt) \ar[r]^-{\Theta_n} & \opt}$$
in which 
\begin{enumerate}
\item[-] the natural transformation $\prod_{i=1}^n a_i^{\bullet}$ is induced by $a_1^{\bullet},\cdots, a_n^{\bullet}$, and 
\item[-] $\Theta_n$ is furnished by the $\Xi^2$-structure on $\opt$ (we have such a structure by Proposition~\ref{mcc_smith_proposition}).
\end{enumerate}
The natural transformation $\beta_n(\alpha, (a_1^{\bullet}, \cdots, a_n^{\bullet})) \colon \Delta^{\bullet} \lra \opt$ is then defined to be the composite
$$\beta_n(\alpha, (a_1^{\bullet}, \cdots, a_n^{\bullet}))= \Theta_n \circ \prod_{i=1}^n a_i^{\bullet} \circ \alpha.$$

One can interpret a $\Xi^2$-structure in another way. McClure and Smith \cite{mcc_smith04} show that the following definition is equivalent to Definition~\ref{xi2_algebra_defn}. 

\begin{defn}
 An augmented cosimplicial space $X^{\bullet}$ is equipped with a \emph{$\Xi^2$-structure} if for each map $f \colon T \lra \overline{n}$ (here $T$ is a totally ordered set and $n \geq 0$) with complexity $\leq 2$, there exists a map 
$$ \crf\colon X^{f^{-1}(1)} \times \cdots \times X^{f^{-1}(n)} \longrightarrow X^T$$
 such that the collection of maps $\{ \crf\}$ is \emph{consistent} (see \cite[Definition 9.4]{mcc_smith04}), \emph{commutative} (see \cite[Definition 9.5]{mcc_smith04}), \emph{associative} (see \cite[Definition 9.6]{mcc_smith04}) and \emph{unital} (see \cite[Definition 9.7]{mcc_smith04}).
\end{defn}

We are going to interpret (in the new language of the $\Xi^2$-structure on $\opt$) the action of $\dtwo$ on $\mbox{Tot} \opt$. We need this interpretation because it will be used in the proof of Theorem~\ref{compatibilty_thm} below. Let $\alpha$, $a_1^{\bullet}, \cdots, a_n^{\bullet}$ as before. We want to construct a natural transformation 
$$\{\beta_n(\alpha, (a_1^{\bullet},\cdots, a_n^{\bullet}))_k \colon \Delta^k \lra \mcalo^k \}_{k \geq 0}.$$ 
So let $[k]$ be an object of the category $\Delta$, and let $t \in \Delta^k$. By Definition~\ref{dtwo_definition},  $\alpha_k$ is a map from $\Delta^k$ to $\Xi^2_n(\Delta^{\bullet}, \cdots, \Delta^{\bullet})([k])$. Recalling that this latter space is the colimit of certain $A_{n[k]}$-diagram (see (\ref{xi2_definition}) above), there exists an object 
$$ \xymatrix{[k] & T \ar[l]_-{h} \ar[r]^-{f} & \overline{n}}$$
in the category $A_{n[k]}$ such that $\alpha_k(t)$ is the equivalence class of some $\widetilde{\alpha}_k(t) \in \Delta^{f^{-1}(1)} \times \cdots \times \Delta^{f^{-1}(n)}$. That is, 
$$\alpha_k(t)=[\widetilde{\alpha}_k(t)].$$
Define  $\beta_n(\alpha, (a_1^{\bullet},\cdots, a_n^{\bullet}))_k(t)$ to be the image of $\widetilde{\alpha}_k(t)$ under the composite
$$\xymatrix{\Delta^{f^{-1}(1)} \times \cdots \times \Delta^{f^{-1}(n)} \ar[rrr]^-{(a_1^{f^{-1}(1)}, \cdots, a_n^{f^{-1}(n)})} \ar[rrrrd] & & & \mcalo^{f^{-1}(1)} \times \cdots \times \mcalo^{f^{-1}(n)} \ar[r]^-{\crf} & \mcalo^T \ar[d]^{h_*} \\
    & & & &    \mcalo^k,}$$
where $h_*$ is the map induced by $h$, and $\crf$ is given by the $\Xi^2$-structure of $\opt$, which is itself induced by the multiplicative structure of the operad $\mcalo$ (we will recall the construction of $\crf$ \cite[Section 10]{mcc_smith04} in the following lines).	That is,
\begin{eqnarray} \label{beta_definition}
\beta_n(\alpha, (a_1^{\bullet},\cdots, a_n^{\bullet}))_k(t)=h_* \circ \crf \circ (a_1^{f^{-1}(1)}, \cdots, a_n^{f^{-1}(n)}) (\widetilde{\alpha}_k(t)).
\end{eqnarray}	
It is straightforward to check that the map $\beta_n(\alpha, (a_1^{\bullet},\cdots, a_n^{\bullet}))_k \colon \Delta^k \lra \mcalo^k$ is well defined. It is also straightforward to check that the collection of maps $\{\beta_n\}_{n \geq 0}$ defines an action of $\dtwo$ on $\mbox{Tot} \opt$.

We now recall the construction of $\crf$. Let $\mu \in \mcalo(2)$ as in the proof of Theorem~\ref{action_cacti_operad}, and let us denote the operad structure of $\mcalo$ by 
$$\gamma \colon \mcalo(n) \times \mcalo(i_1) \times \cdots \times \mcalo(i_n) \lra \mcalo(i_1+\cdots + i_n).$$
\begin{enumerate}
\item[$\bullet$] If $f \colon [r+s+1] \lra \overline{2}$ is defined by the word $f=\underbrace{1 \cdots 1}_{r+1} \underbrace{2 \cdots 2}_{s+1}$, then $\crf \colon \mcalo^r \times \mcalo^s \lra \mcalo^{r+s+1}$ is defined by the formula
\begin{eqnarray} \label{crf1_eq}
\crf(x, y)= \mu(x, d^0y).
\end{eqnarray}
\item[$\bullet$] If $f \colon [2n+i_1+\cdots +i_n] \lra \overline{n+1}$ is on the form $f=1\underbrace{2 \cdots 2}_{i_1+1} 1\underbrace{3 \cdots 3}_{i_2+1} 1 \cdots 1\underbrace{n+1 \cdots n+1}_{i_n+1}1,$
then $\crf \colon \mcalo^n \times \mcalo^{i_1} \times \cdots \times \mcalo^{i_n} \lra \mcalo^{2n+i_1+\cdots i_n}$ is defined by the formula
\begin{eqnarray} \label{crf2_eq}
\crf(x, y_1, \cdots , y_n)= \gamma(x, d^0d^{i_1+1}y_1, \cdots, d^0d^{i_n+1}y_n).
\end{eqnarray}
\item[$\bullet$] For a general $f \colon T \lra \overline{n}$ of complexity less than or equal to $2$, the map $\crf \colon \mcalo^{f^{-1}(1)} \times \cdots \times \mcalo^{f^{-1}(1)} \lra \mcalo^T$ is defined (as a "combination" of formulas (\ref{crf1_eq}) and (\ref{crf2_eq})) by induction on $\left\|T\right\|= \left|T\right|-1$. We refer the reader to \cite[Section 10]{mcc_smith04} for that induction.   
\end{enumerate}

\section{Proof of Theorem~\ref{main_thm}} \label{main_result_section}

The goal of this section is to prove Theorem~\ref{main_thm} announced in the introduction.

Recalling the definition of the  cacti operad $MS$ from Section~\ref{cacti_section}, we have the following theorem in which the second part is a more precise formulation of Theorem~\ref{main_thm}.

\begin{thm} \label{compatibilty_thm}
\begin{itemize}
\item[(a)] There exists an isomorphism $q\colon MS \stackrel{\cong}{\longrightarrow} \mathcal{D}_2$.
 \item[(b)] Let $\opt$ be a cosimplicial space defined by a multiplicative operad $\mathcal{O}$. Then, for each $n \geq 0$, the square
 $$\xymatrix{MS(n) \times (\emph{Tot} \mathcal{O}^{\bullet})^n \ar[rr]^-{\theta_n} \ar[d]_{q_n \times id^n}&& \emph{Tot} \mathcal{O}^{\bullet} \ar[d]_{id} \\ 
\mathcal{D}_2(n) \times (\emph{Tot} \mathcal{O}^{\bullet})^n \ar[rr]^-{\beta_n}&& \emph{Tot} \mathcal{O}^{\bullet}}$$
commutes.
\end{itemize}
\end{thm}

\begin{proof} 
\textbf{Proof of (a)}. This part was proved in \cite{sal10} by P. Salvatore. We will still recall the explicit construction  of $q \colon MS \lra \dtwo$ since we need it to prove part (b). To see that $q$ is an isomorphism of operads, we will refer the reader to \cite[Proposition 8.2]{sal10}.

For each $n \geq 0$, we will construct an isomorphism  $q_n \colon MS(n) \longrightarrow \mathcal{D}_2(n)$ in such a way that the collection $q=\{q_n\}_{n \geq 0} \colon MS \lra \dtwo$ turns out to be a morphism of operads. So let $n \geq 0$ be an integer.\\
If $n=0$ then $q_0 \colon *=MS(0) \longrightarrow \mathcal{D}_2(0)=*$ is the unique map from  the one point space to itself.\\
If $n=1$ then, by (\ref{ms(n)_definition}), we have 
$$MS(1)=\{\varphi \colon S^1 \longrightarrow S^1 \ \mbox{such that}\  \varphi \; \mbox{is weakly monotone}\ \mbox{and}\ \varphi(*)=*\}.$$
 We also have (by Definition~\ref{dtwo_definition}) 
$$\mathcal{D}_2(1)=\mbox{Nat}(\Delta^{\bullet}, \Xi_1^2(\Delta^{\bullet}))=\mbox{Nat}(\Delta^{\bullet}, \Delta^{\bullet})=\mbox{Tot} \Delta^{\bullet}.$$ 
Let $\varphi \in MS(1)$, and let $t=(-1 \leq t_1 \leq \cdots \leq t_k \leq 1) \in \Delta^k$. Define $q_1(\varphi) \colon \Delta^k \lra \Delta^k$ by 
$$q_1(\varphi)(t)=(-1 \leq \varphi(t_1) \leq  \cdots \leq \varphi(t_k) \leq 1).$$
Now we assume that $n \geq 2$. Let $(x, \fitilde) \in MS(n)= \mathcal{I}_2(n) \times \mbox{Mon}(I, \partial I)$. Set $x=(I_1, \cdots, I_n)$. Our goal is to construct $$q_n(x, \fitilde) \in \mathcal{D}_2(n)=\mbox{Nat}(\Delta^{\bullet}, \Xi_k^2(\Delta^{\bullet},\cdots, \Delta^{\bullet})).$$
Let $[k] \in \Delta$. We want to build a map 
$$G_x \colon \Delta^k \lra \Xi_n^2(\Delta^{\bullet}, \cdots, \Delta^{\bullet})([k]).$$ 
So let $t=(-1 \leq t_1 \leq \cdots \leq t_k \leq 1) \in \Delta^k$. Consider families $\{J_i\}_{i=0}^m$ and $\{L_l\}_{l=0}^{m}$ of closed intervals as defined in the beginning of the proof of Theorem~\ref{action_cacti_operad}. Let us take back  the diagram (see (\ref{htf}))  
$$\xymatrix{[k] & [m] \ar[l]_-{h} \ar[r]^-{f} & \overline{n}}.$$  
Clearly we have $\mbox{cplx}(f) \leq 2$ (this is because $\mbox{cplx}(x) \leq 2$ by Definition~\ref{cplxofx_defn}), and  $h$ is a morphism in the category $\Delta_+$. Hence, the triple $(h, [m], f)$ is an object in the category $A_{n[k]}$. Recalling  that (by (\ref{xi2_definition}) above) 
 $$\Xi_n^2(\Delta^{\bullet}, \cdots, \Delta^{\bullet})([k])= \mbox{colim}_{A_{n[k]}} F_{n[k]},$$
we are going now to built an explicit  element $G_x([k])(t)$ of the space 
$$F_{n[k]}(h, [m], f)=\prod_{i=1}^n \Delta^{\left|f^{-1}(i) \right|-1}.$$
Let $i \in \overline{n}$. As in the proof of Theorem~\ref{action_cacti_operad}, we construct an element $y_i \in \Delta^{\left|f^{-1}(i)\right|-1}$ (see (\ref{yi_definition}) for the definition of $y_i$). We thus have an element
$$y= (y_1, \cdots, y_n) \in \prod_{i=1}^n \Delta^{\left|f^{-1}(i) \right|-1}.$$ 
Define now $G_x([k])(t) \in \mbox{colim}_{A_{n[k]}} F_{n[k]}$ to be the equivalence class of $y$. That is, 
$$G_x([k])(t)=[(y_1, \cdots, y_n)].$$
It is straightforward to check that the family $G_x=\{G_x([k])\}_{k \geq 0}$ is a natural transformation. The map $q_n  \colon MS(n) \lra \dtwo(n)$ is then defined by
$$q_n(x, \fitilde)= G_x.$$
It is also straightforward to check that the map $q=\{q_n\}_{n \geq 0} \colon MS \lra \dtwo$ respects the operad structure. 

\textbf{Proof of (b)}. The result follows immediately when $n=0$.\\ 
Let $n \geq 1$, $a_1^{\bullet}, \cdots, a_n^{\bullet} \in \mbox{Tot} \mathcal{O}^{\bullet}$ and $(x, \widetilde{f}) \in MS(n)$. We want to show that 
$$\beta_n(G_{x}, (a_1^{\bullet}, \cdots, a_n^{\bullet}))= \theta_n((x, \fitilde), (a_1^{\bullet}, \cdots, a_n^{\bullet})) \in \mbox{Tot} \opt.$$ 
To do that, we will prove the following equality (for each $k \geq 0$)
$$\beta_n(G_x, (a_1^{\bullet}, \cdots, a_n^{\bullet}))_k= \theta_n((x, \fitilde), (a_1^{\bullet}, \cdots, a_n^{\bullet}))_k \colon \Delta^k \lra \mcalo^k.$$
Let $[k] \in \Delta$. If $k=0$ then the desired equality follows.\\
 Now take $k \geq 1$, and let $t=(t_1, \cdots, t_k) \in \Delta^k$. We have 
$$\begin{array}{lll}
\beta_n(G_x, (a_1^{\bullet}, \cdots, a_n^{\bullet}))_k(t) & = & h_* \circ \crf \circ (a_1^{f^{-1}(1)}, \cdots, a_n^{f^{-1}(n)}) (y)\ \mbox{by (\ref{beta_definition})}\\
  & = &  h_* \circ \crf(a_1^{f^{-1}(1)}(y_1), \cdots, a_n^{f^{-1}(n)}(y_n)).
\end{array}
$$
To end the proof of this part, it suffices to get the following equality 
\begin{eqnarray} \label{crf=theta}
h_* \circ \crf(a_1^{f^{-1}(1)}(y_1), \cdots, a_n^{f^{-1}(n)}(y_n)) = \theta_n((x, \fitilde), (a_1^{\bullet}, \cdots, a_n^{\bullet}))_k(t)=\theta_n(f)
\end{eqnarray}
when  
$$f=\underbrace{1 \cdots 1}_{r+1} \underbrace{2 \cdots 2}_{s+1} \qquad \mbox{and} \qquad f=1\underbrace{2 \cdots 2}_{i_1+1} 1\underbrace{3 \cdots 3}_{i_2+1} 1 \cdots 1\underbrace{n+1 \cdots n+1}_{i_n+1}1.$$

\begin{enumerate}
\item[$\bullet$] If $f=\underbrace{1 \cdots 1}_{r+1} \underbrace{2 \cdots 2}_{s+1}$ then the map $h$ in the diagram $\xymatrix{[r+s] & [r+s+1] \ar[l]_-h \ar[r]^-{f} & \overline{2}}$ is equal to the codegeneracy map $s^{r+1}$. Recalling that $id \in \mcalo(1)$ is the identity operation, we first have
$$ \begin{array}{lll} 
\crf(a_1^r(y_1), a_2^s(y_2)) & = & \mu(a_1^r(y_1), d^0a_2^s(y_2)) \; \; \mbox{by (\ref{crf1_eq})}\\
                            & = & \mu(a_1^r(y_1), \mu(id, a_2^s(y_2))) \in \mcalo^{r+s+1} \; \; \mbox{by the definition of the coface map $d^0$}.
		\end{array}												
$$
Next, recalling that $e \in \mcalo(0)$ is the distinguish operation in arity $0$, we have 
$$ \begin{array}{lll}
                 h_*(\crf(a_1^r(y_1), a_2^s(y_2))) & = & s^{r+1}(\crf(a_1^r(y_1), a_2^s(y_2))) \; \; \mbox{since $h_*=s^{r+1}$} \\
																									& = & (\mu(a_1^r(y_1), \mu(id, a_2^s(y_2)))) \circ_{r+1} e \; \; \mbox{by the definition of $s^{r+1}$}\\
																									& = & \mu(a_1^r(y_1),  a_2^s(y_2)) \\
																									& = & \theta_2(f) \in \mcalo^{r+s} \;\; \mbox{by (\ref{theta2_defn1})},
   \end{array}
$$
thus giving the equality (\ref{crf=theta}).
 
\item[$\bullet$] Now we assume that $f$ is on the form $f=1\underbrace{2 \cdots 2}_{i_1+1} 1\underbrace{3 \cdots 3}_{i_2+1} 1 \cdots 1\underbrace{n+1 \cdots n+1}_{i_n+1}1$. We will prove the equality (\ref{crf=theta}) when $n=3$ and $f=1\underbrace{2 \cdots 2}_{3+1} 1\underbrace{3 \cdots 3}_{5+1} 1\underbrace{4 \cdots 4}_{1+1}1$ for example (the proof being the same for general $n$ and $f$).\\ By the formula (\ref{crf2_eq}), we have
$$\crf(a_1^3(y_1), a_2^3(y_2), a_3^5(y_3), a_4^1(y_4))=\gamma(a_1^3(y_1), d^0d^4a_2^3(y_2), d^0d^6a_3^5(y_3), d^0d^2a_4^1(y_4)) \in \mcalo^{15}$$
Since $f$ is on the above form, it follows that the map $h$ in the diagram $\xymatrix{[9] & [15] \ar[l]_-{h} \ar[r]^-{f} & \overline{4}}$ is defined by the word $h=0012333456788899$. It is not difficult to see that $h=s^{15} \circ s^{13} \circ s^{12} \circ s^6 \circ s^5 \circ s^1$. Therefore 
$$ \begin{array}{lll} 
h_*(\crf(a_1^3(y_1), a_2^3(y_2), a_3^5(y_3), a_4^1(y_4))) & = & \gamma(a_1^3(y_1), a_2^3(y_2), a_3^5(y_3), a_4^1(y_4))\\
                                                          & = & ((a_1^3(y_1) \circ_3 a_4^1(y_4)) \circ_2 a_3^5(y_3)) \circ_1 a_2^3(y_2) \\
																													& = & \theta_4(f) \in \mcalo^9 \; \; \mbox{by  (\ref{thetan_formula})},
\end{array}
$$
\end{enumerate}
thus completing the proof.
\end{proof}

\textsf{Université catholique de Louvain, Chemin du Cyclotron 2, B-1348 Louvain-la-Neuve, Belgique\\
Institut de Recherche en Mathématique et Physique\\}
\textit{E-mail address: arnaud.songhafouo@uclouvain.be}

\end{document}